\documentclass[12pt]{amsart}

\title[Three counterexamples to a conjecture of Colin de Verdi\`ere]{Three counterexamples to a conjecture of Colin de Verdi\`ere on multiplicity}
\author[M. Fortier Bourque]{Maxime Fortier Bourque}
\address{D\'epartement de math\'ematiques et de statistique, Universit\'e de Montr\'eal, 2920, chemin de la Tour, Montr\'eal (QC), H3T 1J4, Canada}
\email{maxime.fortier.bourque@umontreal.ca}

\author[\'E. Gruda-Mediavilla]{\'Emile Gruda-Mediavilla}
\address{D\'epartement de math\'ematiques et de statistique, Universit\'e de Montr\'eal, 2920, chemin de la Tour, Montr\'eal (QC), H3T 1J4, Canada}
\email{emile.gruda-mediavilla@umontreal.ca}

\author[B. Petri]{Bram Petri}
\address{Institut de Math\'ematiques de Jussieu--Paris Rive Gauche and  Institut universitaire de France ; Sorbonne Universit\'e and Universit\'e Paris Cit\'e, CNRS, IMJ-PRG, F-75005 Paris, France;}
\email{bram.petri@imj-prg.fr}

\author[M. Pineault]{Mathieu Pineault}
\address{D\'epartement de math\'ematiques et de statistique, Universit\'e de Montr\'eal, 2920, chemin de la Tour, Montr\'eal (QC), H3T 1J4, Canada}
\email{mathieu.pineault.1@umontreal.ca}

\date{\today}

\usepackage{amsmath,amsfonts,amssymb,amsthm,graphicx,overpic,tablefootnote}
\usepackage{float,enumitem}
\usepackage{color,xcolor,subcaption}
\usepackage[colorlinks=true,linkcolor=red]{hyperref}

\usepackage[margin=2.7cm]{geometry}  
\usepackage[indent]{parskip}

\numberwithin{equation}{section}

\newtheorem{thm}{Theorem}[section]
\newtheorem{prop}[thm]{Proposition}
\newtheorem{prp}[thm]{Proposition}

\newtheorem{lem}[thm]{Lemma}

\theoremstyle{definition}

\theoremstyle{remark}
\newtheorem{rem}[thm]{Remark}
\newtheorem*{acknowledgements}{Acknowledgements}
\newtheorem*{funding}{Funding}

\newcommand{\thmref}[1]{Theorem~\ref{#1}}

\newcommand{\propref}[1]{Proposition~\ref{#1}}

\newcommand{\lemref}[1]{Lemma~\ref{#1}}

\newcommand{\figref}[1]{Figure~\ref{#1}}

\newcommand{\nc}{\newcommand}
\nc{\dmo}{\DeclareMathOperator}
\usepackage{dsfont}

\nc{\wtilde}{\widetilde}
\nc{\abs}[1]{\left| #1 \right|}
\nc{\bigO}[1]{O\left(#1\right)}
\nc{\card}[1]{\left|#1\right|}
\nc{\ceil}[1]{\left\lceil #1 \right\rceil}
\nc{\CC}{\mathbb{C}}
\nc{\dilog}{\mathcal{L}}
\nc{\floor}[1]{\left\lfloor #1 \right\rfloor}
\nc{\ind}{\mathds{1}}
\nc{\ZZ}{\mathbb{Z}}
\nc{\len}[1]{\left| #1 \right|}
\nc{\littleo}[1]{o\left(#1\right)}
\dmo{\Mat}{Mat}
\nc{\NN}{\mathbb{N}}
\nc{\norm}[1]{\left|\left| #1 \right|\right|}
\nc{\QQ}{\mathbb{Q}}
\nc{\RR}{\mathbb{R}}
\nc{\st}[2]{\left\{\, #1 \,:\, #2\,\right\}}
\dmo{\supp}{supp}
\nc{\tr}[1]{\mathrm{tr}\left(#1\right)}
\nc{\what}{\widehat}
\dmo{\im}{Im}
\dmo{\re}{Re}
\nc{\eps}{\varepsilon}
\dmo{\li}{li}

\dmo{\arccosh}{arccosh}
\dmo{\arcsinh}{arcsinh}
\dmo{\area}{area}
\dmo{\conv}{conv}
\dmo{\diam}{diam}
\dmo{\DD}{\mathbb{D}}
\dmo{\dist}{\mathrm{d}}
\nc{\HH}{\mathbb{H}}
\dmo{\Isom}{Isom}
\dmo{\MCG}{MCG}
\dmo{\MPL}{MPL}
\dmo{\Mod}{\mathcal{M}}
\dmo{\PL}{PL}
\nc{\Sphere}{\mathbb{S}}
\dmo{\sys}{sys}
\dmo{\kiss}{kiss}
\dmo{\Teich}{\mathcal{T}}
\nc{\Torus}{\mathbb{T}}
\dmo{\vol}{vol}
\dmo{\WP}{WP}
\nc{\Nsmall}{N_\mathrm{small}}
\dmo{\rect}{rect}
\dmo{\chr}{chr}
\dmo{\spec}{spec}
\dmo{\Irr}{Irr}
\dmo{\axis}{axis}

\dmo{\convTV}{\;\stackrel{\mathrm{TV}}{\longrightarrow}\;}
\nc{\ExV}[2]{\mathbb{E}_{#1}\left[#2\right]}
\dmo{\EE}{\mathbb{E}}
\nc{\Pro}[2]{\mathbb{P}_{#1}\left[#2\right]}
\dmo{\PP}{\mathbb{P}}
\nc{\distTV}[2]{\mathrm{d}_{\rm TV}\left(#1,#2\right)}
\dmo{\UU}{\mathbb{U}}
\nc{\Var}[2]{\mathbb{V}\mathrm{ar}_{#1}\left[#2\right]}

\dmo{\alt}{\mathfrak{A}}
\dmo{\Aut}{Aut}
\dmo{\Fix}{Fix}
\dmo{\GL}{GL}
\dmo{\Hom}{Hom}
\dmo{\id}{id}
\dmo{\PSL}{PSL}
\dmo{\PGL}{PGL}
\dmo{\PO}{PO}
\dmo{\Rep}{Rep}
\dmo{\SL}{SL}
\dmo{\SO}{SO}
\dmo{\sym}{\mathfrak{S}}
\dmo{\inv}{\mathcal{I}}
\dmo{\orb}{\mathcal{O}}
\dmo{\stab}{Stab}

\nc{\calA}{\mathcal{A}}
\nc{\calB}{\mathcal{B}}
\nc{\calC}{\mathcal{C}}
\nc{\calD}{\mathcal{D}}
\nc{\calE}{\mathcal{E}}
\nc{\calF}{\mathcal{F}}
\nc{\calG}{\mathcal{G}}
\nc{\calH}{\mathcal{H}}
\nc{\calI}{\mathcal{I}}
\nc{\calJ}{\mathcal{J}}
\nc{\calK}{\mathcal{K}}
\nc{\calL}{\mathcal{L}}
\nc{\calM}{\mathcal{M}}
\nc{\calN}{\mathcal{N}}
\nc{\calO}{\mathcal{O}}
\nc{\calP}{\mathcal{P}}
\nc{\calQ}{\mathcal{Q}}
\nc{\calR}{\mathcal{R}}
\nc{\calS}{\mathcal{S}}
\nc{\calT}{\mathcal{T}}
\nc{\calU}{\mathcal{U}}
\nc{\calV}{\mathcal{V}}
\nc{\calW}{\mathcal{W}}
\nc{\calX}{\mathcal{X}}
\nc{\calY}{\mathcal{Y}}
\nc{\calZ}{\mathcal{Z}}

\nc{\KK}{\mathbb{K}}
\nc{\FF}{\mathbb{F}}

\begin{document}

\begin{abstract}
We exhibit closed hyperbolic surfaces of genus $10$,  $17$, and $37$ such that the multiplicity of the first nonzero eigenvalue of their Laplacian is larger than the maximum conjectured by Yves Colin de Verdi\`ere in 1986. In order to determine these multiplicities, we apply the twisted Selberg trace formula to the representations induced by the isometry groups of these surfaces on corresponding triangle groups. 
\end{abstract}

\maketitle

\section{Introduction}

For a closed, connected, Riemannian manifold $X$ of dimension at least $1$, let $\lambda_1(X)$ be the smallest nonzero eigenvalue of its Laplacian acting on smooth real-valued functions and let $m_1(X)$ be its multiplicity, that is, the dimension of the corresponding eigenspace. Given a closed, connected, smooth manifold $\Sigma$ of dimension at least $1$, consider
\[\overline{m}_1(\Sigma):= \sup \left\{ m_1(X)  : X \text{ is a Riemannian manifold homeomorphic to }\Sigma \right\}.\]

In \cite{CdV86} and \cite{CdV87}, Colin de Verdi\`ere conjectured that $\overline{m}_1(\Sigma) = \chr(\Sigma) - 1$, where the \emph{chromatic number} $\chr(\Sigma)$ is defined as the supremum of the natural numbers $n$ such that the complete graph on $n$ vertices embeds in $\Sigma$. By results from \cite{Heawood}, \cite{RY}, and \cite{4color1,4color2}, every finite graph embeddable in a closed surface $\Sigma$ can be colored with at most $\chr(\Sigma)$ colors so that adjacent vertices have different colors, hence the name.

 Colin de Verdi\`ere's conjecture is true if $\dim(\Sigma)= 1$, if $\dim(\Sigma)\geq 3$ \cite{CdV86}, or if $\Sigma$ is the $2$-sphere \cite{ChengMultiplicity}, the $2$-torus \cite{Besson}, the projective plane \cite{Besson}, or the Klein bottle \cite{CdV87,NadirashviliMultiplicite}. Many efforts were made to try to prove the conjecture in general \cite{Besson2,Sevennec,LetrouitMachado}, but to no avail. We will show that the conjecture is in fact false.

\begin{thm} \label{thm:main}
There exist closed, connected, orientable, hyperbolic surfaces $X_{10}$, $X_{17}$, and $X_{37}$ of genus $10$, $17$, and $37$ respectively, satisfying $m_1(X_{10})=16> 13 = \chr(X_{10})-1$, $m_1(X_{17})=21>16 = \chr(X_{17})-1$, and $m_1(X_{37})\geq 24 > 23 = \chr(X_{37})-1$.
\end{thm}

Note that the chromatic number of a closed connected surface $\Sigma$ different from the Klein bottle is $\left \lfloor \frac12\left( 7 + \sqrt{49 - 24 \chi(\Sigma)} \right) \right \rfloor$ by a result of Ringel and Youngs \cite{RY}, where $\chi$ is the Euler characteristic. The chromatic number of the Klein bottle is equal to $6$ rather than $7$. In particular, if $\Sigma$ is orientable of genus $g$, then its chromatic number is equal to $\left \lfloor \frac12\left( 7 + \sqrt{48 g +1} \right) \right \rfloor$, which yields the right-hand side equalities in Theorem \ref{thm:main}.

The surfaces $X_{10}$ and $X_{17}$ in \thmref{thm:main} are $(2,3,8)$- and $(2,3,7)$-triangle surfaces which happen to be normal covers of the Bolza and Klein surfaces in genus $2$ and $3$, respectively. They were discovered by \'Emile Gruda-Mediavilla with the help of Mathieu Pineault during a summer research project supervised by Maxime Fortier Bourque.  The goal of the project was to compute $\lambda_1$ and $m_1$ numerically for all triangle surfaces of genus $2$ to $20$ using the computer program \texttt{FreeFEM++}. The largest approximate multiplicity we found in each genus is shown in \figref{fig:mult} together with Colin de Verdi\`ere's conjectured maximum and the upper bound from \cite{LPbounds} (valid for hyperbolic surfaces only).

The surface $X_{37}$ is also a $(2,3,8)$-triangle surface. It was discovered later on using the techniques developed in this paper to rigorously compute the multiplicity of the first nonzero Laplacian eigenvalue for $X_{10}$ and $X_{17}$. Curiously, Riemann surfaces of genus $10$, $17$, and $37$ also appear in \cite{crystallographic}, but we do not know if they have any connection with $X_{10}$, $X_{17}$, and $X_{37}$.

\begin{figure}
\includegraphics[scale=0.75]{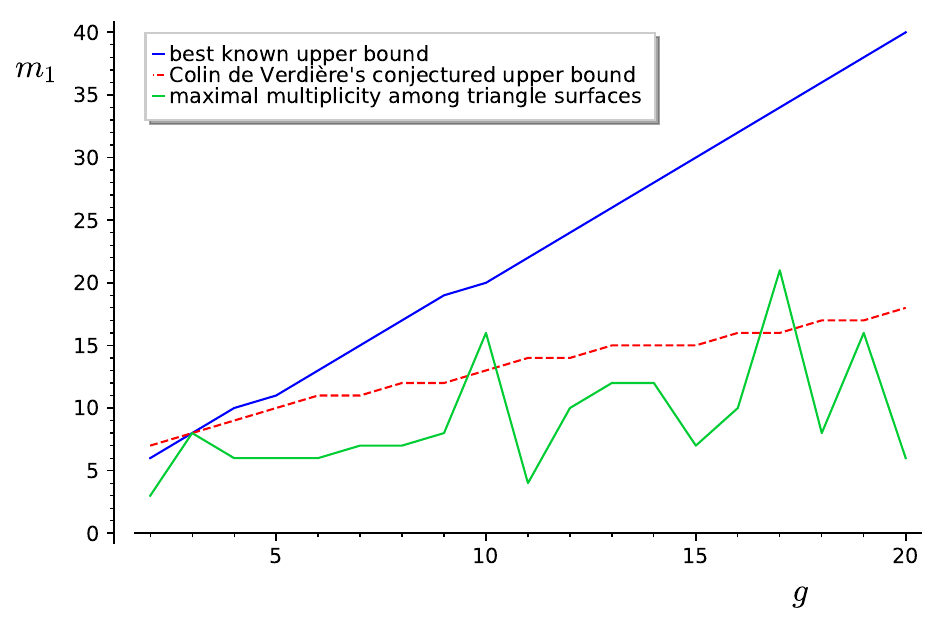}
\caption{Upper and lower bounds on the largest multiplicity of $\lambda_1$ among hyperbolic surfaces of genus $g$ between $2$ and $20$. The green curve is based on numerical results that have been verified rigorously in genus $2$ and $3$ \cite{Klein}, 7 \cite{Lee}, $10$ and $17$ (present paper).}  \label{fig:mult}
\end{figure}

\subsection{Outline of proof}

We now describe how we prove that $m_1(X_{10})$ and $m_1(X_{17})$ are as large as observed numerically. These techniques also show that $m_1(X_{37})\geq 24$ and will be applied in a forthcoming paper to confirm the remaining multiplicities in \figref{fig:mult}. 

In \cite{Klein}, it was shown that the Bolza surface $B$ and the Klein quartic $K$ satisfy $m_1(B)=3$ and $m_1(K)=8$ using the representation theory of their isometry groups and the Selberg trace formula. The idea is that the eigenspaces of the Laplacian can be decomposed into direct sums of irreducible representations (irreps) of the isometry group (or any of its subgroups). If the irreducible representations of small dimension can be ruled out from appearing in the eigenspace corresponding to $\lambda_1$, that forces $m_1$ to be large. The proof of \thmref{thm:main} is based on the same general idea but uses improved methods to exclude representations.

Here is an elementary way to rule out most of the irreps for $X_{10}$ and $X_{17}$. We will not use this argument in the actual proof, so we skip some details. The surfaces $X_{10}$ and $X_{17}$ happen to be chiral, meaning that all their isometries are orientation-preserving. The dimensions of the irreducible representations of $\Isom^+(X_{10})$ and $\Isom^+(X_{17})$ over $\CC$ are
\[
1, 1, 2, 2, 2, 3, 3, 4, 8, 8, 16 \quad \text{and} \quad
1, 3, 3, 6, 7, 7, 7, 8, 14, 21, 21
\]
(see Appendix \ref{app_chartables}) while the corresponding dimensions for $\Isom^+(B)$ and $\Isom^+(K)$ are
\[
1, 1, 2, 2, 2, 3, 3, 4
\quad \text{and} \quad
1, 3, 3, 6, 7, 8
.\]
There is a reason why the second lists are subsets of the first. Since the normal orbifold cover $X_{10}\to \Isom^+(X_{10})\backslash X_{10}$ factors as $X_{10} \to B \to  \Isom^+(B) \backslash B$, there is a surjective homomorphism $\Isom^+(X_{10})\to\Isom^+(B)$. As such, any irreducible representation of $\Isom^+(B)$ induces an irreducible representation of $\Isom^+(X_{10})$ of the same dimension via precomposition with this epimorphism. Suppose that an irrep $(\rho,V)$ of $\Isom^+(X_{10})$ induced by $\Isom^+(B)$ appears in an eigenspace of the Laplacian on $X_{10}$. Then the eigenfunctions in $V$ descend to eigenfunctions on $B$ with the same eigenvalue because the kernel of the homomorphism $\Isom^+(X_{10})\to\Isom^+(B)$ is precisely the deck group of the cover $X_{10}\to B$. The same statements hold for $X_{17}$ and $K$. It folllows that any non-trivial eigenvalue of the Laplacian on $X_{10}$ or $X_{17}$ coming from such induced representations must therefore be at least as large as $\lambda_1(B)\approx 3.838$ or $\lambda_1(K)\approx 2.678$ respectively. However, previous upper bounds from \cite[Table 3]{LPbounds} imply that $\lambda_1(X_{10})\leq 1.223$ and $\lambda_1(X_{17})\leq 0.969$. Hence, this already rules out most of the irreducible representations from appearing in the first eigenspace. For $X_{10}$ there are only two irreps of dimension $8$ left to exclude and for $X_{17}$ there are two of dimension $7$ and one of dimension $14$.

How does one rule out these additional irreps? In \cite{Klein} and \cite{LPbounds}, Courant's nodal theorem was used to exclude $1$-dimensional irreps under certain conditions and some ad hoc methods were used to exclude $2$-dimensional irreps for the Bolza surface, following \cite{Cook} and \cite{JenniThesis}. The Selberg trace formula was also used to prove bounds on the number of eigenvalues in certain intervals for the Klein quartic, yielding the result that $m_1(K)=8$. However, this method would not work here since there could be two nearby eigenvalues whose multiplicities add up to $16=8+8$ or $21=7+14$.

The additional ingredient used here is the \emph{twisted} Selberg trace formula. This version of the trace formula yields information on the eigenvalues associated with individual re\-presentations of the group of orientation-preserving isometries. By applying it to the irreducible re\-presentations of non-maximal dimension with suitable test functions, we show that the corresponding eigenvalues are strictly larger than the first positive eigenvalue of the surface. This implies that $m_1(X_{10}) \geq 16$ and $m_1(X_{17}) \geq 21$, and previous upper bounds on multiplicity imply that equality holds in both cases. Note that this method rules out all the irreducible representations except the top-dimensional ones, so the above argument on induced representations is not needed in the end.

\subsection{Notes and comments}

We conclude  the introduction by pointing out that although Colin de Verdi\`ere's conjecture turns out to be false, it would be interesting to know if it is asymptotically true. In particular, is it true that $\overline{m}_1(\Sigma) = O(\sqrt{|\chi(\Sigma)|}+1)$? 

The best known constructions \cite{BurgerColbois,CCV} achieve this rate of growth with hyperbolic surfaces. To achieve higher multiplicity than in \cite{BurgerColbois} using representation theory, one would need to build regular covers of a base orbifold where several large-dimensional irreps of the deck group are forced to appear together in the first eigenspace. Just one irrep  is not enough because their dimension is bounded by the square root of the cardinality of the group, which is proportional to the Euler characteristic.

The best known upper bound on $\overline{m}_1$  without restrictions on the metric when $\chi(\Sigma)<0$ is $5 - \chi(\Sigma)$ \cite{Sevennec}. This was improved to $2g-1$ for hyperbolic surfaces of large enough genus $g$ in \cite{LPbounds} and a sublinear upper bound was obtained in \cite{LetrouitMachado} for pinched negatively curved surfaces with injectivity radius bounded away from zero.

\begin{acknowledgements}
We thank Jennifer Paulhus and Aaron Wooton for suggesting to look at other genera of the form $g=k^2+1$, which led to the discovery of $X_{37}$. We also thank the \texttt{FreeFEM++} \cite{FreeFEM}, \texttt{SageMath} \cite{sagemath}, \texttt{Arb} \cite{Arb}, and \texttt{GAP} \cite{GAP} developers for making the calculations in this paper possible.
\end{acknowledgements}

\begin{funding}
MFB was partially supported by a Discovery Grant from the Natural Sciences and Engineering Research Council of Canada (NSERC), EGM was supported by an Undergraduate Student Research Award from NSERC, and MP was supported by a Master's Training Scholarship from the Fonds de Recherche du Québec Nature et Technologies and a Canada Graduate Scholarship from NSERC.
\end{funding}

\section{Spectral theory of co-compact Fuchsian groups}\label{sec_spectral}

In this section, we review some of the spectral theory of cocompact Fuchsian groups. Our discussion is mostly based on \cite{Hejhal_vol1} and \cite{TwistedLaplacians}, to which we refer for details and a more complete overiew. Our standing assumptions will be:
\begin{itemize}
\item $\Gamma < \PSL(2,\RR)$ is a co-compact (but not necessarily torsion-free) Fuchsian group,
\item $\Lambda \trianglelefteq \Gamma$ is a normal subgroup of finite index, 
\item $G = \Gamma / \Lambda$ is the quotient group,
\item $\rho:\Gamma \to G$ is the quotient map.
\end{itemize}
These hypotheses imply that there is a normal branched cover $\Lambda\backslash\HH^2 \to \Gamma\backslash\HH^2$ with deck group $G$. In the cases that interest us, $\Lambda\backslash\HH^2$ will be a triangle surface, $G$ its group of orientation-preserving isometries, and $\Gamma$ the corresponding triangle group, so that the quotient $\Gamma\backslash\HH^2$ is the double of a hyperbolic triangle across its boundary.

\subsection{Twisted Laplacians}

Let $C^{\infty}(\Lambda\backslash\HH^2)$ be the space of smooth $\CC$-valued $\Lambda$-invariant functions on $\HH^2$, let $\Delta$ be the Laplacian acting on such functions, and let $\spec(\Lambda\backslash\HH^2)$ denote the spectrum of $\Delta$ as a multiset, where eigenvalues are listed with multiplicity.

Usually, the Laplacian is taken to act on real-valued functions, but the definition extends to complex-valued functions by linearity. Since the Laplacian is self-adjoint and commutes with complex conjugation, its eigenvalues are real and its eigenspaces admits bases consisting of real-valued eigenfunctions. It follows that the spectrum of $\Delta$ acting on complex-valued functions is the same as when acting on real-valued functions. We work over $\CC$ as this is more convenient for representation theory. 

The group $G$ acts on $\Lambda \backslash \HH^2$ by isometries so that the induced action on $C^\infty(\Lambda\backslash\HH^2)$ by precomposition commutes with the Laplacian. This implies that the eigenspaces of the Laplacian can be decomposed into irreducible representations of $G$. We will describe this decomposition using twisted Laplacians. Since the quotient of $\Lambda \backslash \HH^2$ by $G$ is $\Gamma \backslash \HH^2$, the objects in question will be defined over that smaller quotient. 

Let $V$ be a finite-dimensional vector space over $\CC$ equipped with a Hermitian inner product and let $\varphi : \Gamma \to \GL(V)$
be a unitary representation. We write $C^\infty(\Gamma\backslash\HH^2,\varphi)$ for the space of smooth functions $F:\HH^2\to V$ such that
\[
F(\gamma(z)) = \varphi(\gamma)(F(z)), \quad \text{for all }z\in\HH^2 \text{ and all } \gamma\in\Gamma.
\]
We can extend the definition of the Laplacian to $C^\infty(\Gamma\backslash\HH^2,\varphi)$ by making it act coordinate-wise with respect to an orthonormal basis for $V$ and this \emph{twisted Laplacian} is denoted by $\Delta_\varphi$. The hypothesis that $\varphi$ is unitary is used so that $\varphi(\gamma)$ is an isometry, hence commutes with $\Delta_\varphi$, so that $\Delta_\varphi F$ satisfies the same equivariance property as $F$, i.e., belongs to $C^\infty(\Gamma\backslash\HH^2,\varphi)$ as well. We write $\spec(\Gamma\backslash\HH^2,\varphi)$ for the spectrum of $\Delta_\varphi$.

If $\varphi = \phi\circ \rho$ is induced by a unitary representation $\phi:G \to \GL(V)$ via the quotient map $\rho: \Gamma \to G$, then we write $\spec(\Gamma\backslash\HH^2,\phi)$ instead of $\spec(\Gamma\backslash\HH^2,\varphi)$. We will also write $\dim_\CC(\phi):=\dim_\CC(V)$ for the degree of the representation $\phi$.

\subsection{The decomposition}

Let $\Irr(G)$ be the set of (equivalence classes of) irreducible representations of $G$ over $\CC$, chosen to be unitary (which is always possible by Weyl's trick). The decomposition of $\spec(\Lambda\backslash\HH^2)$ into irreducible representations of $G$ can be written as follows. 

\begin{prop}\label{prp_specdecomp}
If $\Lambda \trianglelefteq \Gamma$ are co-compact Fuchsian groups and $G = \Gamma/\Lambda$ is finite, then
\[
\spec(\Lambda\backslash\HH^2) = \bigcup_{\phi \in \Irr(G)} \dim_\CC(\phi) \cdot  \spec(\Gamma\backslash\HH^2,\phi)
\]
as multisets. 
\end{prop}

Here, multiplying a multiset by $d$ means multiplying the multiplicities by $d$. For a proof of this proposition, see for instance \cite[Sections 3.8 and 3.9]{TwistedLaplacians}. Technically, our set-up is slightly different from that of Cornelissen and Peyerimhoff: they consider the spectrum of the Laplacian on a Riemannian manifold. However, their proof goes through for orbifolds as well. The only minor difference is how the spaces $C^\infty(\Lambda\backslash\HH^2)$ and $C^\infty(\Gamma\backslash\HH^2,\varphi)$ are defined.

\subsection{The twisted Selberg trace formula}

In order to access the spectra appearing on the right-hand side of Proposition \ref{prp_specdecomp}, we will use the twisted Selberg trace formula. To state it, we need some further notation.

Given an integrable function  $f:\RR\to\CC$, its Fourier transform $\what{f}$ is given by 
\[
\what{f}(y) = \int_{-\infty}^\infty f(x)\; e^{-i y \cdot x} \,dx.
\]
We will call $f$ \emph{admissible} if it is even and there exists an $\eps>0$ such that $\what{f}$ is holomorphic on a strip $\calS_\eps = \{z\in \CC;\;\abs{\mathrm{Im}(z)} < \frac{1}{2}+\eps\}$ and 
\[
\what{f}(y) = O\left( (1+\abs{y})^{-2-\eps}\right)
\]
in $\calS_\eps$. 

If $\Gamma < \PSL(2,\RR)$ is a cocompact Fuchsian group, $\varphi : \Gamma \to \GL(V)$ is a finite-dimensional unitary representation, and $f$ is an admissible function, then the \emph{twisted Selberg trace formula} (see \cite[p. 351]{Hejhal_vol1}) states that
\begin{align}\label{eq_STF} 
\sum_{\lambda \in \spec(\Gamma\backslash\HH^2,\varphi)} \what{f}\left(\sqrt{\lambda - \frac{1}{4}}\right) &=  \dim_\CC(\varphi) \frac{\mathrm{area}(\Gamma\backslash \HH^2)}{4\pi} \int_{-\infty}^\infty y \what{f}(y) \tanh(\pi y)\,dy \\ \notag
 &\quad + \sum_{[\gamma] \in \calE(\Gamma)} \frac{\tr{\varphi(\gamma)}}{2 m(\gamma) \sin(\theta(\gamma))} \int_{-\infty}^\infty \frac{e^{-2\theta(\gamma) y}}{1+e^{-2\pi y}} \what{f}(y)\,dy \\ \notag
  &\quad + \sum_{[\gamma] \in \calH'(\Gamma)} \ell(\gamma)\; \sum_{n\geq 1} \frac{\tr{\varphi(\gamma^n)}}{2\sinh(n\ell(\gamma)/2)} f(n\ell(\gamma))
\end{align}

Here,
\begin{itemize}
\item $\calE(\Gamma)$ denotes the set of conjugacy classes of elliptic elements in $\Gamma$,
\item for an elliptic element $\gamma\in\Gamma$, $m(\gamma)$ denotes the order of the centralizer of $\gamma$ in $\Gamma$ (which in this case is the largest cyclic subgroup containing $\gamma$) and $\theta(\gamma)$ denotes half the angle of rotation of $\gamma$, i.e., is such that $\gamma$ is conjugate to
\[
\left[ \begin{array}{cc} \cos(\theta(\gamma)) & \sin(\theta(\gamma)) \\ - \sin(\theta(\gamma)) & \cos(\theta(\gamma)) \end{array}\right] \text{ in } \PSL(2,\RR),
\]
\item  $\calH'(\Gamma)$ denotes the set of conjugacy classes of primitive hyperbolic elements in $\Gamma$,
\item for a hyperbolic element $\gamma\in\Gamma$, $\ell(\gamma)$ denotes its translation length on $\HH^2$.
\end{itemize}
Moreover, on the left-hand side (the spectral side) we may use any of the two branches of the square root, because $\what{f}$ is even.

We will write $\calG(\Gamma,\phi,f)$ for the right-hand side (the geometric side) of the twisted Selberg trace formula \eqref{eq_STF} in the case that $\varphi$ is induced by $\phi: G \to \GL(V)$.

\subsection{Simplifications in the trace formula}\label{sec_simplified_terms}

It is possible to express the geometric side of the trace formula entirely in terms of the function $f$ (instead of its Fourier transform $\what{f}$). In our application, $f$ will be compactly supported, so writing the geometric side in terms of it will make the integrals that appear easier to bound rigorously. 

The identity and elliptic terms on the geometric side of \eqref{eq_STF} can be written as:
\[
\int_{-\infty}^\infty y \what{f}(y)\tanh(\pi y) dy = -\int_{-\infty}^\infty \frac{f'(x)}{\sinh(x/2)}dx 
\]
and
\[
\frac{1}{2\sin(\theta)}\int_{-\infty}^\infty \frac{e^{-2\theta\cdot y}}{1+e^{-2\pi y}} \what{f}(y) dy = \int_0^\infty  \frac{\cosh(x/2)}{\cosh(x)-1+2 \sin(\theta)^2} f(x) dx,
\]
see \cite[p. 27-28, 450]{Hejhal_vol1}.

\section{Excluding representations}

In this section, we explain how to prove lower bounds on the eigenvalues of the Laplacian twisted by an irreducible representation $\phi$. If $\spec(\Gamma\backslash \HH^2),\phi)$ is disjoint from an interval $I$, then $\phi$ cannot appear in any eigenspace of $\Delta$ on $\Lambda\backslash \HH^2$ corresponding to an eigenvalue in $I$ in view of \propref{prp_specdecomp}. We will apply this criterion in Section \ref{sec:proof} to exclude all but the top-dimensional irreducible representations from appearing in the eigenspaces corresponding to $\lambda_1(X_{10})$, $\lambda_1(X_{17})$, and $\lambda_1(X_{37})$, where $X_{10}$, $X_{17}$, and $X_{37}$ are the surfaces from Theorem \ref{thm:main}. 

\subsection{The criterion}

Our main tool for excluding representations is derived from the twisted Selberg trace formula. A similar criterion is behind the method of Booker--Str\"ombergsson \cite{BookerStrombergsson} to which we will come back below.  We still assume that $\Lambda \trianglelefteq \Gamma$ are co-compact Fuchsian groups and $G=\Gamma/\Lambda$ is finite as in Section \ref{sec_spectral}.

\begin{prp}\label{prp_linineq}
Let $\phi\in\Irr(G)$, let $\lambda>0$, and suppose that there exists an admissible function $f:\RR\to\RR$ such that $\what{f}\left(\sqrt{\mu-\frac{1}{4}}\right) \geq 0$ for all $\mu \geq 0$ and
\[
\what{f}\left(\sqrt{\lambda-\frac{1}{4}}\right) > \begin{cases}
\calG(\Gamma,\phi,f) & \text{if }\phi \text{ is non-trivial} \\
\calG(\Gamma,\phi,f)-\what{f}(i/2) & \text{if }\phi \text{ is trivial.}
\end{cases}
\]
Then $\lambda \not\in \spec(\Gamma\backslash\HH^2,\phi).$
\end{prp}

\begin{proof}
Suppose for a contradiction that $\lambda \in \spec(\Gamma\backslash\HH^2,\phi)$. By the twisted Selberg trace formula \eqref{eq_STF} and the non-negativity hypothesis on $\what{f}$, we have  
\[
\what{f}\left(\sqrt{\lambda-\frac{1}{4}}\right) \leq \sum_{\mu \in \spec(\Gamma\backslash\HH^2,\phi)} \what{f}\left(\sqrt{\mu - \frac{1}{4}}\right) = \calG(\Gamma,\phi,f),
\]
which is a contradiction if $\phi$ is non-trivial. If $\phi$ is trivial, then $0 \in \spec(\Gamma\backslash\HH^2,\phi)$ so that we have
\[
\what{f}(i/2)+\what{f}\left(\sqrt{\lambda-\frac{1}{4}}\right) \leq \sum_{\mu \in \spec(\Gamma\backslash\HH^2,\phi)} \what{f}\left(\sqrt{\mu - \frac{1}{4}}\right) = \calG(\Gamma,\phi,f),
\]
which is again a contradiction.
\end{proof}

In practice, we will use the same function $f$ for all $\lambda$ in some interval $(0,b]$, which will prove that $\spec(\Gamma\backslash\HH^2,\phi) \cap (0,b] = \varnothing$. Also note that the hypotheses imply that the spectral side of the trace formula is real, hence we can take the real part of the character associated to $\phi$ in order to compute the geometric side $\calG(\Gamma,\phi,f)$.

\subsection{Test functions}

Given $d>0$, we define the following admissible pair
\[
f_d(x) = \left(\frac{1}{2d}\chi_{[-d,d]}\right)^{\ast 4}(x) \quad \text{and} \quad \what{f_d}(y) = \frac{\sin(d\cdot y)^4}{(d\cdot y)^4},
\]
where $\chi_{[-d,d]}$ denotes the characteristic function of the interval $[-d,d]$ and the exponent $\ast 4$ denotes the fourfold convolution product of  the function with itself. An elementary computation yields that
\[
f_d(x) = \left\{
\begin{array}{ll}
\frac{1}{12d} \left(4 - \frac{3}{2 d^2} x^2 + \frac{3}{8 d^3} \abs{x}^3 \right) & \text{if } 0  \leq \abs{x} \leq 2d, \\[3mm]
\frac{1}{12d}  \left(2 - \frac{\abs{x}}{2d}\right)^3 & \text{if } 2d \leq \abs{x} \leq 4d \text{ and}\\[3mm]
0 & \text{otherwise.}
\end{array}\right.
\]
In particular,
\[
\supp(f_d) = [-4d,4d].
\]

The method of Booker--Str\"ombergsson \cite{BookerStrombergsson} is based on functions that are linear combinations of shifts of test functions of this type. This method is very effective for estimating the spectrum of hyperbolic orbifolds and has been applied in various other contexts since \cite{LL1,LL2,LL3,BonifacioMazacPal,GPSDX}. We initially applied this method here as well, but then realized that just one test function was enough for our purposes.

Observe that $\what{f_d}$ is non-negative on the real and imaginary axes. Furthermore, it is increasing along the positive imaginary axis and decreasing on the interval $[0,\pi/d]$, so that $\what{f_d}\left(\sqrt{\lambda-1/4}\right)$ is decreasing for $\lambda \in [0,(\pi/d)^2+1/4]$. It follows that if the inequality in Proposition \ref{prp_linineq} is satisfied at $\lambda = b \leq (\pi/d)^2+1/4$, then it is also satisfied for every $\lambda \in (0,b]$.

The fact that $f_d$ has compact support is useful for us, because it means that if we know all the elliptic conjugacy classes and all the conjugacy classes of primitive hyperbolic elements of translation length at most $4d$ in our Fuchsian group $\Gamma$, then we can compute the geometric side $\calG(\Gamma,\phi,f_d)$ in the twisted Selberg trace formula. Since $X_{10}$, $X_{17}$, and $X_{37}$ are triangle surfaces, the corresponding Fuchsian groups (that will play the role of $\Lambda$) are normal subgroups of triangle groups (that will play the role of $\Gamma$). In that case, the conjugacy classes of elliptic elements are simply the conjugacy classes of the standard generators $x,y,z$ (see Section \ref{sec_trianglegroups}) and their powers. In the next section, we will explain how to list the conjugacy classes of primitive hyperbolic elements of small translation length in these groups.

\section{Generating conjugacy classes}\label{sec_gen_conj_classes}

Given a cocompact Fuchsian group $\Gamma < \PSL(2,\RR)$ and a real number $L>0$, we need an algorithm that lists the distinct conjugacy classes of primitive hyperbolic elements in $\Gamma$ with translation length at most $L$. Such an algorithm is described in \cite{HodgsonWeeks} for hyperbolic $3$-orbifolds and is used in the computer program \texttt{SnapPy}. The same algorithm works in dimension $2$ as well. We describe this algorithm with minor modifications below.

The idea is to generate enough elements in $\Gamma$ so that the images of a fundamental domain by these elements cover a ball of a certain radius $R=R(L)$ around a basepoint $x_0 \in \HH^2$. We then check for conjugacy between all pairs of hyperbolic elements with translation length at most $L$ via the elements in the previous list. Finally, we delete the conjugacy classes of non-primitive elements by checking if they are (conjugate to) powers of other elements. Throughout, we use interval arithmetic to make numerical calculations rigorous.

\subsection{The word problem}

To test for conjugacy, we first need to be able to test whether two elements $f,g \in\Gamma$ are equal. In theory, this should be easy: for any fixed pair of distinct points $x_0,x_1 \in \HH^2$, we have that $f=g$ if and only if they agree on $x_0$ and $x_1$ because the only element in $\PSL(2,\RR)$ with more than one fixed point in $\HH^2$ is the identity. 

The problem comes when we try to do this on the computer, where arbitrary real numbers cannot be represented exactly. Instead, we represent elements of $\Gamma$ as $2$ by $2$ matrices whose entries lie in certain intervals with rational endpoints. All calculations performed subsequently keep track of correct intervals containing the answers by rounding upper and lower bounds appropriately. This is called \emph{interval arithmetic}. 

The drawback is that we only know elements of $\Gamma$ up to a certain precision, but the redeeming feature is that $\Gamma$ is discrete, so we can still tell elements apart if the precision is good enough. To be precise, fix two distinct points $x_0$ and $x_1$ and let $\delta_j$ be the smallest distance between two distinct points in the $\Gamma$-orbit of $x_j$. If $f, g \in \Gamma$, we want to determine if $f(x_0)=g(x_0)$ and $f(x_1)=g(x_1)$, which is equivalent to
\[
d(f(x_0),g(x_0)) < \delta_0 \quad \text{and} \quad d(f(x_1),g(x_1)) < \delta_1.
 \] 
On the other hand, if either distance is strictly positive, then $f \neq g$. In theory, if there is too much imprecision on $f(x_j)$ or $g(x_j)$, then the computer might not be able to decide if the above inequalities hold or not, but we have not encountered that possibility in practice.

We will apply the above criterion when $\Gamma$ is a triangle group, where we take $x_0$ and $x_1$ to be fixed points of elliptic elements and $\delta_0$ and $\delta_1$ can be computed explicitly.

\subsection{The conjugacy problem} \label{subsec:conjugacy}

Let $F \subset \HH^2$ be a compact fundamental domain for $\Gamma$. If $\gamma \in \Gamma$ is hyperbolic, we denote its translation axis (the geodesic between its two fixed points at infinity) by $\axis(\gamma)$. If $h \in \Gamma$, then observe that $\axis(h \gamma h^{-1}) = h(\axis(\gamma))$. Since for every $y \in \HH^2$ there is some $h \in \Gamma$ such that $h(y) \in F$, any hyperbolic element $\gamma \in \Gamma$ has a conjugate $h \gamma h^{-1}$ whose axis intersects $F$. If $D$ is the diameter of $F$ and $x_0 \in F$ is some basepoint, then in particular the axis of $h \gamma h^{-1}$ is within distance $D$ of $x_0$. In \cite{HodgsonWeeks}, Hodgson and Weeks make this more efficient by using the largest distance from $x_0$ to an edge of $F$ instead of the diameter, assuming that $F$ is a Dirichlet fundamental domain.

Although there are more hyperbolic elements whose axis passes within distance $D$ of $x_0$ than whose axis intersects $F$, the former condition is easier to test.

\begin{lem} \label{lem:dist_to_axis}
Let $\gamma \in \PSL(2,\RR)$ be a hyperbolic element with translation length $\ell$, let $x_0 \in \HH^2$, and let $\delta$ be the distance between $x_0$ and $\axis(\gamma)$. Then \[\sinh(d(x_0,\gamma(x_0))/2)=\sinh(\ell / 2) \cosh(\delta).\]
\end{lem}
\begin{proof}
Form a Saccheri quadrilateral with $x_0$, $\gamma(x_0)$ and their orthogonal projections onto $\axis(\gamma)$, then divide it into two congruent Lambert quadrilaterals along its axis of symmetry. The resulting Lambert quadrilateral has a short side of length $\ell/2$, the opposite side of length $d(x_0,\gamma(x_0))/2$, and the longer of the two remaining sides of length $\delta$. The equation stated is Formula 2.3.1(v) in \cite[p.454]{Buser} applied to this quadrilateral.
\end{proof}

Note that the translation length $\ell$ of a hyperbolic element $\gamma \in \PSL(2,\RR)$ is easily determined in terms of its absolute trace, so the above formula gives a way to compute the distance $\delta$ between $x_0$ and the axis of $\gamma$ without actually finding that axis, which would be cumbersome to do rigorously on the computer.

By Lemma $\ref{lem:dist_to_axis}$, to generate all hyperbolic elements in $\Gamma$ of translation length at most $L$ whose axes pass within distance $D$ from $x_0$, it suffices to generate all the elements that move $x_0$ by distance at most
\[
r = 2 \arcsinh(\sinh(L/2)\cosh(D)).
\]
In other words, it suffices to generate a set $E \subset \Gamma$ such that $\bigcup_{\gamma \in E}  \gamma(F)$ contains the closed ball of radius $r$ around $x_0$. Whether that inclusion holds is something we can test by computing the minimum distance from $x_0$ to any boundary side of $\bigcup_{\gamma \in E}  \gamma(F)$ (which requires knowing which sides of the polygons $\gamma(F)$ are on the boundary of the union). Note that our formula for $r$ is smaller than the one used in \cite{HodgsonWeeks} because we are working in dimension $2$. In dimension $3$, a hyperbolic element can act as a screw motion along an axis, which increases the distance between $x_0$ and $\gamma(x_0)$.

The next thing we need is a stopping criterion to test for conjugacy between hyperbolic elements in the set $E$.

\begin{lem} \label{lem:conjugacy}
Suppose that $\alpha, \beta \in \Gamma$ are conjugate hyperbolic elements of translation length $\ell$ whose axes intersect the closed ball of radius $D$ around $x_0 \in \HH^2$. Then there is some $h \in \Gamma$ such that $\beta = h \alpha h^{-1}$ and
\[
d(x_0, h(x_0)) \leq 2\arccosh(\cosh(\ell/4)\cosh(D)).
\]
\end{lem}
\begin{proof}
Let $h_0 \in \Gamma$ be an element that conjugates $\alpha$ and $\beta$. Then $h_0(\axis(\alpha)) = \axis(\beta)$. In particular,
\[
d(h_0(x_0),\axis(\beta)) = d(x_0,\axis(\alpha)) \leq D.
\]
Let $z$ be the point on $\axis(\beta)$ closest to $x_0$ and let $w_0$ be the orthogonal projection of $h_0(x_0)$ onto that axis. Observe that we can replace $h_0$ with $h = \beta^k h_0$ for any $k \in \ZZ$ and still have $\beta = h \alpha h^{-1}$. Choose $k$ in such a way that $w=\beta^k(w_0)$ satisfies $d(z,w)\leq \ell/2$. If $m$ is the midpoint between $z$ and $w$, then 
\[
d(x_0,h(x_0)) \leq d(x_0, m) + d(m, h(x_0)).
\]
Each of the last two distances is the hypothenuse of a right triangle with legs of length at most $\ell/4$ and $D$. The stated inequality then follows from the hyperbolic Pythagorean theorem \cite[Equation 2.2.2(i), p.454]{Buser}.
\end{proof}

The upshot of \lemref{lem:dist_to_axis} and \lemref{lem:conjugacy} is that in order to find a unique representative of each conjugacy class of hyperbolic elements of translation length at most $L$ in $\Gamma$, it suffices to generate a set $E$ of elements in $\Gamma$ so that the images of $F$ by these elements cover the ball of radius 
\[
R = \max\Big\{2 \arcsinh(\sinh(L/2)\cosh(D)),2\arccosh(\cosh(L/4)\cosh(D))\Big\}
\] 
around $x_0$. By \lemref{lem:dist_to_axis} and the remarks preceding it, such a set $E$ contains at least one representative of each conjugacy class of hyperbolic elements of translation length at most $L$ in $\Gamma$ whose axis intersects the closed ball of radius $D$ around $x_0$ and by \lemref{lem:conjugacy}, any two such elements that are conjugate in $\Gamma$ are conjugate via an element of $E$.

We explain how to efficiently generate elements of $\Gamma$ in the special case where $\Gamma$ is a $(2,3,r)$-triangle group with $r\geq 7$ in the next subsection.

\subsection{An explicit automatic structure for some triangle groups}\label{sec_trianglegroups}

For integers $p,q,r \geq 2$ with $\frac1p+\frac1q+\frac1r<1$, the \emph{$(p,q,r)$-triangle group $T(p,q,r)$} (unique up to conjugation in $\PSL(2,\RR)$) is the group generated by rotations $x$, $y$, $z$ of counterclockwise angles $\frac{2\pi}{p}$, $\frac{2\pi}{q}$, and $\frac{2\pi}{r}$ around the vertices of a triangle $\tau_0$ in $\HH^2$ with interior angles $\frac{\pi}{p}$, $\frac{\pi}{q}$, and $\frac{\pi}{r}$ appearing in this order counterclockwise around $\tau_0$. This group is discrete and cocompact, and admits the presentation
\[
T(p,q,r) \cong \left \langle x, y, z : x^p = y^q = z^r = xyz = \id \right\rangle.
\]

To generate elements of a group efficiently, it is useful to do it without redundancy, especially when the group has exponential growth. This is precisely what automatic structures are for. Note that every cocompact Fuchsian group is word-hyperbolic and every word-hyperbolic group is automatic (see e.g. \cite[Section 3.4]{Epstein}). There also exist algorithms to compute finite state automata for such groups (see \cite[Chapter 5]{Epstein} and \cite{EpsteinHolt}). However, it is more convenient if automata are readily available. For triangle groups $T(p,q,r)$ with $p,q,r \geq 6$, explicit automatic structures were found in \cite{Pfeiffer}. This does not cover the cases $T(2,3,7)$ and $T(2,3,8)$ that we need. An automaton for $T(2,3,7)$ was described without proof in \cite[Example 1]{CKK} and is used in the video game \texttt{HyperRogue}. We generalize this automaton to groups $T(2,3,r)$ with $r \geq 7$ and provide a proof below. The automaton also works for the triangle group $T(2,3,6)$ acting on the euclidean plane.

Let $A$, $B$, and $C$ be the centers of rotation of $x$, $y$, and $z$ respectively (the vertices of $\tau_0$) and let $F$ be the union of $\tau_0$ with its reflection about the geodesic through $A$ and $C$. This is a fundamental domain for the action of $T(p,q,r)$ in general. It is also true that $F$ is a Dirichlet fundamental domain for any point in the interior of the segment between $A$ and $C$. However, it seems more natural to use $x_0 = C$ as basepoint when applying the results of subsection \ref{subsec:conjugacy} due to the symmetries of the tiling of $\HH^2$ by $(p,q,r)$-triangles.

If $p=2$ (in which case $F$ is an isoceles triangle) and $q=3$, both of which we will assume from hereon, then $P_0 = \bigcup_{k=0}^{r-1} z^k(F)$ is a regular $r$-gon with interior angles $2\pi/3$. The distinct images of $P_0$ by $T_r=T(2,3,r)$ form a tiling $\calP$ of $\RR^2$ or $\HH^2$ and the elements of $T_r$ are in bijection with the oriented edges in this tiling. Concretely, the identity element is associated to the edge $e_0$ of $F$ opposite to the vertex $C$ and oriented so that $F$ is to its left, and then any element $g\in T_r$ is associated to the edge $g(e_0)$. For every $g \in T_r$, there is also an associated polygon $g(P_0)$ on the left of the edge $g(e_0)$, but this map from group elements to polygons is $r$-to-$1$.

Observe that $x(e_0)$ is $e_0$ rotated by $\pi$ around its midpoint (so the same edge with reverse orientation) and $z(e_0)$ is the edge following $e_0$ counterclockwise around $P_0$ (so it is the oriented edge forward and left of $e_0$ in the tiling $\calP$). Suppose that $w \in T_r$ is written as $w = w_1 \cdots w_k$ with $w_j \in \{x,z\}$ for every $j \in \{ 1, \ldots, k \}$. How does one get to the edge $w(e_0)$ starting from $e_0$? By reading the word $w_1 \cdots w_k$ from left to right and interpreting each letter as an instruction to be applied to the current edge, where an $x$ means reversing the orientation of an edge and $z$ means moving to the next edge forward and left in the tiling $\calP$. Indeed, for any edge $g(e_0)$ and either generator $f \in \{ x,z \}$, we have $gf(e_0)=gfg^{-1}(g(e_0))$, which is $g(e_0)$ with reversed orientation if $f=x$ and is the next edge forward and left of $g(e_0)$ if $f=z$ because of the conjugation by $g$. Thus, $T_r$ acts on the right on the set of oriented edges via these instructions. 

In order to generate $T_r$ (or equivalently the oriented edges in the tiling), we start with the oriented edge $e_0$ (corresponding to the identity element) then apply powers of $z$ to obtain the other oriented edges that have $P_0$ on their left. For each of these, we then apply $x$ to flip the edge and move $P_0$ to an adjacent polygon $P$. Once we reach a polygon $P$, we can move around its edges with powers of $z$ and then move to further adjacent polygons with $x$. We do this in layers of polygons that form rings around $P_0$, being careful to avoid collisions or repetitions by using precise rules.

The finite state automaton we use for $T_r$ has four states L(eft), M(iddle), R(ight), I(nterior) in addition to the initial state $\varnothing$. The possible transitions between different states are given in Figure \ref{pic_automaton}. For the purpose of generating the group, one could merge the states I and R into a single one, but it is useful to keep them separate in order to compute the boundary of the region of the tiling covered after some number of iterations. This is used in conjunction with the results of the previous subsection to determine when to stop.

\begin{figure}
\begin{center}
\includegraphics[scale=0.9]{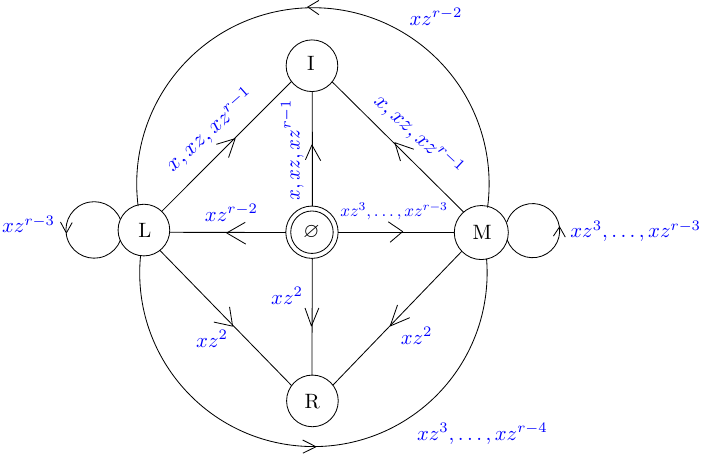}
\caption{A picture of the automaton generating the elements of $T(2,3,r)$. We have drawn the partial automaton (i.e., without dead states). The initial state (center) and the states L, M, R, I are all accepted. Multiplication is on the right and the empty word represents the identity.}\label{pic_automaton}
\end{center}
\end{figure}

Let us describe what the automaton does in words instead of using the formal terminology from the theory of automata. We start with the set $E_1:=\{\id, z, \ldots, z^{r-1}\}$ whose elements are all assigned the label M except for the identity which has the label $\varnothing$. Then for each $n\geq 1$, the next generation $E_{n+1}$ is defined as the union of the children of all the elements $g \in E_n$, where the children of an element $g$ are defined according to the following rules:
\begin{enumerate}
\item if $g$ has label I or R, then $g$ has no children;
\item otherwise, the children of $g$ are the elements of the form $gxz^j$ for $0\leq j \leq r-1$, with the following labels:
\begin{enumerate}
\item $gxz^{r-1}$, $gx$, and $gxz$ have label I;
\item $gxz^2$ has label R;
\item if $g$ has label M or $\varnothing$, then $gxz^j$ has label M if $3 \leq j \leq r-3$ and L if $j=r-2$;
\item if $g$ has label L, then $gxz^j$ has label M if $3 \leq j \leq r-4$, L if $j=r-3$, and I if $j=r-2$.
\end{enumerate}
\end{enumerate}

Observe that the oriented edges $gxz^j(e_0)$ all have the same polygon $gx(P_0)$ to their left, and this is the polygon to the right of $g(e_0)$. The labels have the following meaning. If the polygon in the tiling $\calP$ to the right of an oriented edge was obtained in the previous or the current generation, then that edge gets labelled $I$  because it is in the interior of the union of all the polygons obtained so far. The remaining edges around $gx(P_0)$ are consecutive. We label the first and last one in clockwise order with L (for leftmost) and R (for rightmost) respectively, and the remaining ones with M (because they are in the middle of this string of consecutive exterior edges). Even though the polygon to the right of an oriented edge labelled R has not been obtained yet, it will be obtained in the next generation because it lies on the right of an oriented edge of the current generation with label L. That is why the elements with label R do not generate children.
 
We now prove that this is indeed an automaton for the group $T_r$.

\begin{prop}
For every $r\geq 6$, the above automaton eventually generates every element of the triangle group $T(2,3,r)$ and it does so exactly once. That is,
\[
T(2,3,r) = \bigsqcup_{n\geq 1} E_n.
\]
\end{prop}
\begin{proof}
Since the map from $T_r = T(2,3,r)$ to polygons in the tiling $\calP$ sending $g \in T_r$ to $g(P_0)$ is $r$-to-$1$ and since every element $g$ with descendants has exactly $r$ children, which are of the form $gxz^j$ for $0 \leq j \leq r-1$ hence all of which yield the same polygon, it suffices to check that every polygon in $\calP$ except $P_0$ is obtained exactly once as the \emph{first-born child} $gx(P_0)$ of an element $g$. Note that $g(P_0)$ and $gx(P_0)$ are adjacent along the side $g(e_0)$.

For every $n \geq 1$, let $\calL_n = \bigcup_{g \in E_n} g(P_0)$ be the union of the polygons produced at generation $n$. We will prove by induction that
\begin{enumerate}
\item \label{bulletA} $\calL_n$ is the union of polygons adjacent to the outer boundary of $\calL_{n-1}$;
\item \label{bulletC} $\calL_{n}$ forms an embedded cycle of polygons, meaning that its polygons can be enumerated as $P_1$ to $P_m$ such that $P_j$ is only adjacent to $P_{j-1}$ and $P_{j+1}$ for each $j \in \{1,\ldots, m\}$, where indices are taken modulo $m$. In particular, $\calL_{n}$ is homeomorphic to an annulus.
 
\item \label{bulletB}  the outer boundary of the annulus $\calL_{n}$ consists of the oriented edges of generation $n$ that are labelled L, M, or R;

\item  \label{bulletD} each polygon in $\calL_n$ is the first-born child $gx(P_0)$ of a unique element $g$ in $E_{n-1}$;
\end{enumerate}
for every $n\geq 2$.

The statements are obvious for $n=2$, where $\calL_2$ is formed of the $r$ polygons adjacent to $\calL_1=P_0$. For every $g \in E_1$, the side $gx(e_0)$ of $gx(P_0)$ is on the inner boundary of $\calL_2$ while the sides $gxz^{\pm 1}(e_0)$ are shared with the two neighboring polygons $gz^{\pm 1}x(P_0)$. The remai\-ning sides of $\calL_2$ form a cycle whose labels form the string $(\mathrm{L}\mathrm{M}^{r-5}\mathrm{R})^r$, when read clockwise (see Figure \ref{fig:layers} for an example). Furthermore, the $r$ polygons $gx(P_0)$ for $g \in E_1$ are distinct.

Suppose the above statements are true for $n=k \geq 2$ and let us prove them for $n=k+1$.

\noindent\underline{Proof of \eqref{bulletA}:} By item \eqref{bulletB} for $n=k$, the only active elements in $E_k$ (those with label L or M) are on the outer boundary. This implies that $\calL_{k+1}$ is contained in the union of polygons adjacent to the outer boundary of $\calL_k$. It remains to show that the polygons adjacent to the edges labelled R are also contained in $\calL_{k+1}$. This is true because if $s$ is an edge of $\partial \calL_k$ with label R, then the next edge, say $g(e_0)$, along $\partial \calL_k$ in clockwise order is the first edge of the following polygon that is not labelled I (by \eqref{bulletB}), hence is labelled L. Therefore, the polygon $gx(P_0)$ is produced at generation $k+1$, and this polygon is adjacent to both $g(e_0)$ and $s$ since there are only three polygons around each vertex in $\calP$. This proves that \eqref{bulletA} holds for $n=k+1$.

\noindent\underline{Proof of \eqref{bulletC} and \eqref{bulletB}:} By item \eqref{bulletC} and item \eqref{bulletB} for $n=k$, and by our labelling rules, the labels along the outer boun\-da\-ry of $\calL_{k}$ form a cycle of strings of the form $\mathrm{L}\mathrm{M}^{r-5}\mathrm{R}$ or $\mathrm{L}\mathrm{M}^{r-6}\mathrm{R}$ (see Figure \ref{fig:layers} for an example). Observe that each such string corresponds to the (consecutive) edges of a single polygon $P$ in $\calL_k$ that lie on the outer boundary of $\calL_k$. Since all the vertices in the tiling have degree $3$, the polygons adjacent to $P$ across the sides labelled L and M form a sequence, each adjacent to the next. Furthermore, the polygon $Q$ adjacent to $P$ through the side labelled L contains the edge of the preceding string $s$ that is labelled R. Hence $Q$ is adjacent to the last polygon (the one across the last edge labelled M) in the sequence of polygons corresponding to the string $s$.  This means that the sequences of polygons corresponding to the strings link up to form a closed cycle. However, we still have to show that the polygons $P_1,\ldots,P_m$ in this cycle are pairwise distinct. In other words, we need to show that the cycle is embedded rather than just immersed.

In any case, this immersed cycle yields an immersion $f$ from an annulus $S^1 \times [0,1]$ to $\calL_{k+1}$. By our induction hypothesis, the inner loop, say $f(S^1\times \{0\})$, of that immersed cycle is embedded because it is the outer boundary of the embedded annulus $\calL_k$. We will show that the outer loop $\alpha = f(S^1 \times \{1\})$ is also embedded using geometry.

\begin{figure}
\begin{center}
\includegraphics[scale=0.8]{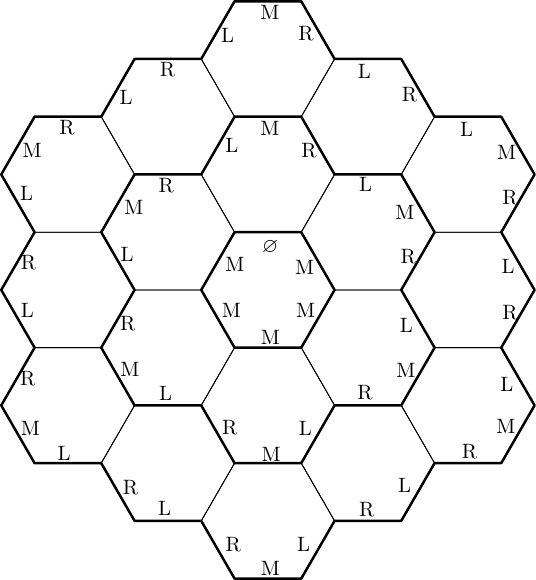}
\caption{The labels assigned to edges on the outer boundary of some layers in the tiling associated with the triangle group $T(2,3,6)$. The remaining edges are all labelled I.}\label{fig:layers}
\end{center}
\end{figure}

For each of the strings $\mathrm{L}\mathrm{M}^{m}\mathrm{R}$ along the outer boun\-da\-ry of $\calL_{k}$, and for each edge $g(e_0)$ with label $\mathrm{M}$ in that string, the edge $gx(e_0)$ is on the inner boundary of $\calL_{k+1}$ while the edges $gxz^{\pm 1}(e_0)$ are shared between consecutive polygons in $\calL_{k+1}$, hence are in its interior. For the edge $g(e_0)$ with label L, $gx(e_0)$ and $gxz^{-1}(e_0)$ are both on the inner boundary ($gxz^{-1}(e_0)$ coincides with the edge labelled R in the preceding string), while $gxz(e_0)$ and $gxz^{-2}(e_0)$ are in the interior. In either case, the children of $g$ labelled $I$ according to our rules are not along the outer loop $\alpha$ of the immersed cycle of polygons $\calL_{k+1}$. Therefore, $\alpha$ consists of the edges of $\calL_{k+1}$ labelled L, M, or R and these still form a cycle of strings of the form $\mathrm{L}\mathrm{M}^{r-5}\mathrm{R}$ or $\mathrm{L}\mathrm{M}^{r-6}\mathrm{R}$. 

Now, observe that the interior angles of $\alpha$ are $2\pi/3$ except at junctions between edges with labels R and L, where the interior angle is $4\pi/3$. Replace each such inward corner by the geodesic between its endpoints. This increases the interior angle at the adjacent vertices, but by not more than $\frac{\pi}{6}+\frac{\pi}{6}$, for a total of at most $\frac{2\pi}{3}+\frac{\pi}{6}+\frac{\pi}{6} = \pi$. The resulting curve is therefore convex, hence embedded, and therefore so is $\alpha$. It follows that the cycle of polygons in $\calL_{k+1}$ is also embedded, completing the proof of \eqref{bulletC} for $n=k+1$. Furthermore, the outer boundary of this embedded annulus $\calL_{k+1}$ is equal to the outer loop $\alpha$. As we had observed above, $\alpha$ consists of the oriented edges in $\calL_{k+1}$ with label L, M, or R, which proves \eqref{bulletB} for $n=k+1$.

\noindent\underline{Proof of \eqref{bulletD}:}  By construction, our cycle of polygons was obtained as the set of images of $P_0$ by the first-born children of the elements in $E_k$ with label L or M. We have just shown that the polygons in this cycle are all distinct, which yields statement \eqref{bulletD} for $n=k+1$ and completes the induction step.

We now show that the set $U =\bigcup_{n=1}^\infty \calL_n$ is equal to the entire euclidean or hyperbolic plane using a standard argument from the theory of tilings. $U$ is open since $\bigcup_{n=1}^m \calL_n$ is in the interior of $\bigcup_{n=1}^{m+1} \calL_n$ by item \eqref{bulletA} above. It is also complete. Indeed, suppose that $(z_k)_{k=1}^\infty$ is a Cauchy sequence in $U$. By shifting the sequence we may assume that its diameter is at most the length of any edge in the tiling $\calP$. Thus, if $z_1 \in \bigcup_{n=1}^m \calL_n$ then the entire sequence is contained in $\bigcup_{n=1}^{m+1} \calL_n$, which is compact, hence complete. We conclude that the sequence  $(z_k)_{k=1}^\infty$ converges in $U$, so that $U$ is complete. In particular, $U$ is closed. Since $\RR^2$ and $\HH^2$ are connected and $U$ is non-empty, $U$ is equal to the whole space.

We conclude that every polygon $P$ in $\calP \setminus P_0$ is contained in some $\calL_n$. By \eqref{bulletA}, the layers  $\calL_j$ all have disjoint interiors, so that $n$ is unique. Moreover, the elements that have $P$ as their first-born child are necessarily contained in $E_{n-1}$. Indeed, $P$ is only adjacent to polygons in $\calL_{n-1}$, $\calL_n$ and $\calL_{n+1}$, so it could only be the first-born child of an element in $E_{n-1}$, $E_n$, or $E_{n+1}$ but by definition the first-born child of an element in $E_k$ yields a polygon contained in $\calL_{k+1}$. Lastly, $P$ has a unique parent in $E_{n-1}$ by \eqref{bulletD}.
\end{proof}

We use this automaton to generate elements in the group $T(2,3,r)$ until we cover a large enough ball around the point $x_0 = C$ to capture all conjugacy classes of hyperbolic elements of translation length at most $L$ (and test for conjugacy between them). In order to compute the inner radius of the region covered after $n$ generations, we use the fact that its boundary is the set of edges in the last generation with labels different from I. We then measure the distance from $C$ to each of these edges and take the minimum. Since the distance function is convex, the distance from $C$ to an edge $e$ is either the height of the triangle with base $e$ and opposite vertex $C$, or the minimum of the lengths of the two other edges, depending on whether the altitude is contained in the triangle or not. 

The resulting algorithm is implemented in the \texttt{Jupyter} notebook \texttt{generate\_classes} (to be used with \texttt{SageMath}) attached as an ancillary file with the arXiv version of this paper. The program outputs lists of primitive conjugacy classes up to any reasonable translation length $L$ for any $r \geq 7$, although the computation time is exponential in $L$. These lists of conjugacy classes are then fed as input in other programs to compute the geometric side of the twisted Selberg trace formula for the irreducible representations of the orientation-preserving isometry groups of the surfaces $X_{10}$, $X_{17}$, and $X_{37}$.

\section{Finishing the proof} \label{sec:proof}

In this last section, we combine the ingredients of the previous sections to prove \thmref{thm:main}. We first describe the surfaces $X_{10}$, $X_{17}$, and $X_{37}$ and recall the strategy of proof. We then explain which part the computer does for us and how these calculations imply the result.

\subsection{The surfaces}

We have not fully described the surfaces $X_{10}$, $X_{17}$, and $X_{37}$ from Theorem \ref{thm:main} yet. They are defined as $X_g=\Lambda_{g}\backslash \HH^2$ for some torsion-free finite-index normal subgroups $\Lambda_{10}, \Lambda_{37} \triangleleft T(2,3,8)$ and $\Lambda_{17} \triangleleft T(2,3,7)$. In turn, these groups are defined as the normal closures
\[
\Lambda_{10} = \langle\langle \; (zyxz)^2yz^{-1}xy^{-1}z^{-2}xz\;\rangle\rangle^{T(2,3,8)},
\]
\begin{multline*}
\Lambda_{17} = \langle\langle \; z^{-3} x y z^{-3},\\ 
 x z y x z y x y^{-1} z^{-1} x y^{-1} z^{-1} x y^{-1} z^{-1} x y^{-1}  z^{-1} x y^{-1} z^{-1} x y^{-1} z^{-1},\\
  z y x z y z^{-1} x z y^{-1} z^{-1} x y^{-1} z^{-2} x z y^{-1} z^{-1} x y^{-1} z^{-2} x z\;\rangle\rangle^{T(2,3,7)},
\end{multline*}
and
\begin{multline*}
\Lambda_{37} = \langle\langle \; x^2, y^3, x y z, z^8, y x z y x z y x z y x y^{-1} z^{-1} x y^{-1} z^{-1} x z,\\ z^2 y z^{-1} x z^2  y  z^{-3}  x  z  y  x^2 y^{-1} z^{-1} x y^{-1} z^{-1} x y^{-1} z^-2 x z y^{-1} z^{-2} x,\\ z y x z y z^{-1} x z y x z^2 y x z y x z^2 y^{-1} z^{-3} x \;\rangle\rangle^{T(2,3,8)},
\end{multline*}
where $x$, $y$, and $z$ are the standard generators of order $p$, $q$, and $r$ in $T(p,q,r)$. We also write $G_{10}:= T(2,3,8)/\Lambda_{10}$, $G_{17}:=T(2,3,7)/\Lambda_{17}$, and $G_{37}:=T(2,3,8)/\Lambda_{37}$.

These groups were taken from the list \cite{ConderList} of all triangle surfaces of genus at most $101$ calculated by Conder using algorithms explained in \cite{ConderPaper}. The groups $G_{10}$, $G_{17}$, $G_{37}$ appear as \texttt{T10.1}, \texttt{T17.1}, and \texttt{T37.1} in Conder's list.

For completeness, we verify that $X_{10}$, $X_{17}$, $X_{37}$ are indeed surfaces without cone points and of the correct genera $10$, $17$, and $37$ in the \texttt{Jupyter} notebooks \texttt{verify\_genus\_10}, \texttt{verify\_genus\_17}, and \texttt{verify\_genus\_37} attached as ancillary files with the arXiv version of this paper. We also use \texttt{SageMath} and \texttt{GAP} to compute the character tables of $G_{10}$, $G_{17}$, and $G_{37}$; the results can be found in Appendix \ref{app_chartables}.

\subsection{Set-up and strategy}

First, we fix the parameters that will appear in the proof. 
\begin{itemize}
\item We set $\lambda_1^{\max}(10) = 1.223$, $\lambda_1^{\max}(17) = 0.969$, and $\lambda_1^{\max}(37)=0.912$. These are upper bounds for the spectral gap in the respective genera given in \cite[Table 3]{LPbounds}. Note that the upper bound we use in genus $37$ is in fact an upper bound in genus $20$ since that is the largest genus we had considered in that paper. These upper bounds on $\lambda_1$ coming from linear programming do not increase as $g$ increases due to the formulation of \cite[Theorem 8.1]{LPbounds}. 
\item We set $\Gamma_{10} = \Gamma_{37}= T(2,3,8)$ and $\Gamma_{17}=T(2,3,7)$.  
\item $\mu_{10} = 16$, $\mu_{17}=21$, and $\mu_{37}=24$ are the lower bounds on multiplicity we want to prove. These are the maximal dimensions of irreps of $G_{10}$, $G_{17}$, and $G_{37}$ over $\CC$ (see Appendix \ref{app_chartables}). 
\item The parameter $d$ we choose equals $L/4$ (so that $f_d$ has support $[-L,L]$), where $L=4$ is the translation length up to which we list the primitive hyperbolic conjugacy classes in $\Gamma_{g}$.
\end{itemize}

In all three cases, we have that $\lambda_1^{\max}(g) < \left(\frac{\pi}{d}\right)^2+\frac14$ so that $\what{f_d}\left(\sqrt{\lambda - \frac{1}{4}}\right)$ is decreasing on the interval $(0,\lambda_1^{\max}(g)]$. Our goal is then to prove that $\what{f_d}\left(\sqrt{\lambda_1^{\max}(g) - \frac{1}{4}}\right)$ is strictly larger than the right-hand side in Proposition \ref{prp_linineq} for all irreducible representations $\phi$ of $G_g$ of dimension less than $\mu_g$. This requires estimating the different terms on the geometric side of the twisted Selberg trace formula with sufficient precision.

\subsection{Formal verification}
We have performed the formal verification of the inequalities we are after using \texttt{SageMath}. The full code is available in the ancillary files \texttt{verify\_genus\_10}, \texttt{verify\_genus\_17}, \texttt{verify\_genus\_37}. We outline the steps here:
\begin{enumerate}
\item We use the interface to \texttt{GAP} in order to determine all the irreducible representations of the groups $G_g$  over $\CC$ for $g\in\{10,17,37\}$.
\item We use interval arithmetic and our lists of primitive hyperbolic conjugacy classes generated by the file \texttt{generate\_classes} to estimate the last sum in $\calG(\Gamma_g,\phi,f_d)$. 

\item For the elliptic and identity terms, we use the formulas in terms of $f_d$ from Section \ref{sec_simplified_terms}; this avoids having to estimate indefinite integrals. The integrals that do appear are treated using interval arithmetic as implemented in the \texttt{Arb} package. The integral for the identity term is slightly trickier because the integrand is a quotient of two functions that vanish at the origin. We use a Taylor approximation with error estimate near zero to get around this issue.

\end{enumerate}

\subsection{The proof} In this final subsection, we gather our results.

\begin{proof}[Proof of Theorem \ref{thm:main}] By the results of our computer code (\texttt{verifiy\_genus\_10}, \linebreak \texttt{verifiy\_genus\_17}, and \texttt{verifiy\_genus\_37}), we have
\[
\what{f_d}\left(\sqrt{\lambda - \frac{1}{4}}\right) \quad > \quad \left\{
\begin{array}{ll}
\calG(\Gamma_g,\phi,f_d) & \text{for all non-trivial }\phi \in \Irr(G_g) \\
& \text{with }\dim_{\CC}(\phi) < \mu_g \\
\calG(\Gamma_g,\phi,f_d)-\what{f_d}(i/2) & \text{for the trivial representation }\phi\in \Irr(G_g)
\end{array}\right.
\]
at $\lambda = \lambda_1^{\max}(g)$ and therefore for all $\lambda \in (0, \lambda_1^{\max}(g)]$ as well, for every $g\in\{10,17,37\}$.

By Proposition \ref{prp_linineq},
this means that $\spec(\Gamma_g \backslash \HH^2,\phi)$ is disjoint from $(0,\lambda_1^{\max}(g)]$ for every $\phi \in \Irr(G_g)$ with $\dim_\CC(\phi)<\mu_g$. Therefore, every eigenvalue in $\spec(\Lambda_g \backslash \HH^2)\cap (0,\lambda_1^{\max}(g)]$ has multiplicity equal to a sum of dimensions $\dim_\CC(\phi)$ for representations $\phi \in \Irr(G_g)$ of complex dimension at least (hence equal to) $\mu_g$ by Proposition \ref{prp_specdecomp}. We also know that $\lambda_1(X_{g}) \in (0,\lambda_1^{\max}(g)]$
by \cite[Table 3]{LPbounds}, so that $m_1(X_{g}) = k_{g} \mu_g$ for some integers $k_{g}\geq 1$. In particular, $m_1(X_{37})\geq 24$. From \cite[Table 5]{LPbounds}, we know that $m_1(X_{10}) \leq 20$ and $m_1(X_{17}) \leq 34$, from which we conclude that $k_{10}=k_{17}=1$, hence $m_1(X_{10}) = 16$ and $m_1(X_{17}) = 21$, as required.
\end{proof}

\begin{rem}
Note that Sévennec's upper bound of $2g+3$ for $m_1(X_g)$ also suffices to show that $k_g \leq 1$ for $g\in\{ 10,17\}$, but cannot rule out the possibility that $m_1(X_{37})=48$. Excluding that possibility can be done with a slight modification of Proposition \ref{prp_linineq}, which implies that in fact $m_1(X_{37})=24$.
\end{rem}

\begin{rem}
For $G_{17}$, our computer program is able to exclude the irreducible representation corresponding to $\chi_{11}$ in Table \ref{table:17}, so it is the irreducible representation corresponding to $\chi_{10}$ that appears in the first eigenspace. Similarly, only the irrep corresponding to $\chi_{19}$ in Table \ref{table:37} can appear in the first eigenspace for $X_{37}$.
\end{rem}

%%%%%%%%%%%%%%%%%%%%%%%%%%%%%%%%%%%%%%%%%%%%%%%%%
%		B I B L I O G R A P H Y
%%%%%%%%%%%%%%%%%%%%%%%%%%%%%%%%%%%%%%%%%%%%%%%%%
\bibliography{biblio}
\bibliographystyle{amsalpha}

\newpage

\appendix

\section{Character tables}\label{app_chartables}

In this appendix, we give the character tables of the orientation-preserving isometry groups $G_{10}$, $G_{17}$, and $G_{37}$ of $X_{10}$, $X_{17}$, and $X_{37}$. These were produced with the interface to \texttt{GAP} of \texttt{SageMath}. The rows of the tables correspond to the characters of the irreducible representations over $\CC$ and the columns correspond to the conjugacy classes in the group. The symbol $\zeta_k$ stands for $e^{2\pi i / k}$. %These computations can be found in the \texttt{Jupyter} notebooks \texttt{verify\_genus\_10} and \texttt{verify\_genus\_17}.

\begin{table}[!h]
\begin{center}
\begin{tabular}{|l||l|l|l|l|l|l|l|l|l|l|l|}
\hline
 & $\{e\}$ & $C_{1}$ & $C_{2}$ & $C_{3}$ & $C_{4}$ & $C_{5}$ & $C_{6}$ & $C_{7}$ & $C_{8}$ & $C_{9}$ & $C_{10}$\\ \hline
\hline
$\chi_{1}$ & $1$ & $1$ & $1$ & $1$ & $1$ & $1$ & $1$ & $1$ & $1$ & $1$ & $1$ \\
\hline
$\chi_{2}$ & $1$ & $-1$ & $1$ & $1$ & $1$ & $1$ & $-1$ & $-1$ & $1$ & $1$ & $-1$ \\
\hline
$\chi_{3}$ & $2$ & $0$ & $-1$ & $2$ & $2$ & $2$ & $0$ & $0$ & $-1$ & $-1$ & $0$ \\
\hline
$\chi_{4}$ & $2$ & $0$ & $-1$ & $0$ & $-2$ & $2$ & $-\zeta_8-\zeta_8^3$ & $0$ & $1$ & $-1$ & $\zeta_8+\zeta_8^3$ \\
\hline
$\chi_{5}$ & $2$ & $0$ & $-1$ & $0$ & $-2$ & $2$ & $\zeta_8+\zeta_8^3$ & $0$ & $1$ & $-1$ & $-\zeta_8-\zeta_8^3$ \\
\hline
$\chi_{6}$ & $3$ & $-1$ & $0$ & $-1$ & $3$ & $3$ & $1$ & $-1$ & $0$ & $0$ & $1$ \\
\hline
$\chi_{7}$ & $3$ & $1$ & $0$ & $-1$ & $3$ & $3$ & $-1$ & $1$ & $0$ & $0$ & $-1$ \\
\hline
$\chi_{8}$ & $4$ & $0$ & $1$ & $0$ & $-4$ & $4$ & $0$ & $0$ & $-1$ & $1$ & $0$ \\
\hline
$\chi_{9}$ & $8$ & $-2$ & $2$ & $0$ & $0$ & $-1$ & $0$ & $1$ & $0$ & $-1$ & $0$ \\
\hline
$\chi_{10}$ & $8$ & $2$ & $2$ & $0$ & $0$ & $-1$ & $0$ & $-1$ & $0$ & $-1$ & $0$ \\
\hline
$\chi_{11}$ & $16$ & $0$ & $-2$ & $0$ & $0$ & $-2$ & $0$ & $0$ & $0$ & $1$ & $0$\\
\hline
\end{tabular}
\medskip

\caption{The character table of $G_{10}$.}
\end{center}
\end{table}

\begin{table}[!h]
\begin{center}
\begin{tabular}{|l||l|l|l|l|l|l|l|l|l|l|l|}
\hline
 & $\{e\}$ & $C_{1}$ & $C_{2}$ & $C_{3}$ & $C_{4}$ & $C_{5}$ & $C_{6}$ & $C_{7}$ & $C_{8}$ & $C_{9}$ & $C_{10}$\\ \hline
\hline
$\chi_{1}$ & $1$ & $1$ & $1$ & $1$ & $1$ & $1$ & $1$ & $1$ & $1$ & $1$ & $1$ \\
\hline
$\chi_{2}$ & $3$ & $3$ & $-1$ & $-1$ & $-1$ & $1$ & $0$ & $0$ & $\zeta_7^3+\zeta_7^5+\zeta_7^6$ & $\zeta_7+\zeta_7^2+\zeta_7^4$ & $1$ \\
\hline
$\chi_{3}$ & $3$ & $3$ & $-1$ & $-1$ & $-1$ & $1$ & $0$ & $0$ & $\zeta_7+\zeta_7^2+\zeta_7^4$ & $\zeta_7^3+\zeta_7^5+\zeta_7^6$ & $1$ \\
\hline
$\chi_{4}$ & $6$ & $6$ & $2$ & $2$ & $2$ & $0$ & $0$ & $0$ & $-1$ & $-1$ & $0$ \\
\hline
$\chi_{5}$ & $7$ & $7$ & $-1$ & $-1$ & $-1$ & $-1$ & $1$ & $1$ & $0$ & $0$ & $-1$ \\
\hline
$\chi_{6}$ & $7$ & $-1$ & $3$ & $-1$ & $-1$ & $1$ & $1$ & $-1$ & $0$ & $0$ & $-1$ \\
\hline
$\chi_{7}$ & $7$ & $-1$ & $-1$ & $-1$ & $3$ & $-1$ & $1$ & $-1$ & $0$ & $0$ & $1$ \\
\hline
$\chi_{8}$ & $8$ & $8$ & $0$ & $0$ & $0$ & $0$ & $-1$ & $-1$ & $1$ & $1$ & $0$ \\
\hline
$\chi_{9}$ & $14$ & $-2$ & $2$ & $-2$ & $2$ & $0$ & $-1$ & $1$ & $0$ & $0$ & $0$ \\
\hline
$\chi_{10}$ & $21$ & $-3$ & $1$ & $1$ & $-3$ & $-1$ & $0$ & $0$ & $0$ & $0$ & $1$ \\
\hline
$\chi_{11}$ & $21$ & $-3$ & $-3$ & $1$ & $1$ & $1$ & $0$ & $0$ & $0$ & $0$ & $-1$\\
\hline
\end{tabular}
\medskip
\caption{The character table of $G_{17}$.}\label{table:17}
\end{center}
\end{table}

\begin{table}
\begin{tabular}{|l||l|l|l|l|l|l|l|l|l|l|l|}
\hline
 & $\{e\}$ & $C_{1}$ & $C_{2}$ & $C_{3}$ & $C_{4}$ & $C_{5}$ & $C_{6}$ & $C_{7}$ & $C_{8}$ & $C_{9}$\\
\hline
\hline
$\chi_{1}$ & $1$ & $1$ & $1$ & $1$ & $1$ & $1$ & $1$ & $1$ & $1$ & $1$ \\
\hline
$\chi_{2}$ & $1$ & $1$ & $1$ & $1$ & $1$ & $-1$ & $-1$ & $1$ & $1$ & $1$ \\
\hline
$\chi_{3}$ & $2$ & $-1$ & $2$ & $2$ & $2$ & $0$ & $0$ & $2$ & $2$ & $2$ \\
\hline
$\chi_{4}$ & $2$ & $-1$ & $2$ & $2$ & $2$ & $0$ & $0$ & $-2$ & $0$ & $0$ \\
\hline
$\chi_{5}$ & $2$ & $-1$ & $2$ & $2$ & $2$ & $0$ & $0$ & $-2$ & $0$ & $0$ \\
\hline
$\chi_{6}$ & $3$ & $0$ & $3$ & $3$ & $3$ & $-1$ & $-1$ & $3$ & $-1$ & $-1$ \\
\hline
$\chi_{7}$ & $3$ & $0$ & $3$ & $3$ & $3$ & $1$ & $1$ & $3$ & $-1$ & $-1$ \\
\hline
$\chi_{8}$ & $3$ & $0$ & $3$ & $-1$ & $-1$ & $-1$ & $-1$ & $-1$ & $-2\zeta_8^2 - 1$ & $2\zeta_8^2 - 1$ \\
\hline
$\chi_{9}$ & $3$ & $0$ & $3$ & $-1$ & $-1$ & $-1$ & $-1$ & $-1$ & $2\zeta_8^2 - 1$ & $-2\zeta_8^2 - 1$ \\
\hline
$\chi_{10}$ & $3$ & $0$ & $3$ & $-1$ & $-1$ & $1$ & $1$ & $-1$ & $-2\zeta_8^2 - 1$ & $2\zeta_8^2 - 1$ \\
\hline
$\chi_{11}$ & $3$ & $0$ & $3$ & $-1$ & $-1$ & $1$ & $1$ & $-1$ & $2\zeta_8^2 - 1$ & $-2\zeta_8^2 - 1$ \\
\hline
$\chi_{12}$ & $4$ & $1$ & $4$ & $4$ & $4$ & $0$ & $0$ & $-4$ & $0$ & $0$ \\
\hline
$\chi_{13}$ & $6$ & $0$ & $6$ & $-2$ & $-2$ & $0$ & $0$ & $-2$ & $2$ & $2$ \\
\hline
$\chi_{14}$ & $6$ & $0$ & $6$ & $-2$ & $-2$ & $0$ & $0$ & $2$ & $0$ & $0$ \\
\hline
$\chi_{15}$ & $6$ & $0$ & $6$ & $-2$ & $-2$ & $0$ & $0$ & $2$ & $0$ & $0$ \\
\hline
$\chi_{16}$ & $8$ & $-1$ & $-1$ & $8$ & $-1$ & $2$ & $-1$ & $0$ & $0$ & $0$ \\
\hline
$\chi_{17}$ & $8$ & $-1$ & $-1$ & $8$ & $-1$ & $-2$ & $1$ & $0$ & $0$ & $0$ \\
\hline
$\chi_{18}$ & $16$ & $1$ & $-2$ & $16$ & $-2$ & $0$ & $0$ & $0$ & $0$ & $0$ \\
\hline
$\chi_{19}$ & $24$ & $0$ & $-3$ & $-8$ & $1$ & $-2$ & $1$ & $0$ & $0$ & $0$ \\
\hline
$\chi_{20}$ & $24$ & $0$ & $-3$ & $-8$ & $1$ & $2$ & $-1$ & $0$ & $0$ & $0$\\
\hline
\end{tabular}

 \medskip

\begin{tabular}{|l||l|l|l|l|l|l|l|l|l|l|l|}
\hline
 & $C_{10}$ & $C_{11}$ & $C_{12}$ & $C_{13}$ & $C_{14}$ & $C_{15}$ & $C_{16}$ & $C_{17}$ & $C_{18}$ & $C_{19}$\\
\hline
\hline
$\chi_{1}$ & $1$ & $1$ & $1$ & $1$ & $1$ & $1$ & $1$ & $1$ & $1$ & $1$ \\
\hline
$\chi_{2}$ & $-1$ & $-1$ & $-1$ & $-1$ & $1$ & $1$ & $1$ & $-1$ & $-1$ & $1$ \\
\hline
$\chi_{3}$ & $0$ & $0$ & $0$ & $0$ & $2$ & $-1$ & $-1$ & $0$ & $0$ & $2$ \\
\hline
$\chi_{4}$ & $-\zeta_8^3 - \zeta_8$ & $-\zeta_8^3 - \zeta_8$ & $\zeta_8^3 + \zeta_8$ & $\zeta_8^3 + \zeta_8$ & $-2$ & $-1$ & $1$ & $0$ & $0$ & $0$ \\
\hline
$\chi_{5}$ & $\zeta_8^3 + \zeta_8$ & $\zeta_8^3 + \zeta_8$ & $-\zeta_8^3 - \zeta_8$ & $-\zeta_8^3 - \zeta_8$ & $-2$ & $-1$ & $1$ & $0$ & $0$ & $0$ \\
\hline
$\chi_{6}$ & $1$ & $1$ & $1$ & $1$ & $3$ & $0$ & $0$ & $-1$ & $-1$ & $-1$ \\
\hline
$\chi_{7}$ & $-1$ & $-1$ & $-1$ & $-1$ & $3$ & $0$ & $0$ & $1$ & $1$ & $-1$ \\
\hline
$\chi_{8}$ & $\zeta_8^2$ & $-\zeta_8^2$ & $\zeta_8^2$ & $-\zeta_8^2$ & $3$ & $0$ & $0$ & $1$ & $1$ & $1$ \\
\hline
$\chi_{9}$ & $-\zeta_8^2$ & $\zeta_8^2$ & $-\zeta_8^2$ & $\zeta_8^2$ & $3$ & $0$ & $0$ & $1$ & $1$ & $1$ \\
\hline
$\chi_{10}$ & $-\zeta_8^2$ & $\zeta_8^2$ & $-\zeta_8^2$ & $\zeta_8^2$ & $3$ & $0$ & $0$ & $-1$ & $-1$ & $1$ \\
\hline
$\chi_{11}$ & $\zeta_8^2$ & $-\zeta_8^2$ & $\zeta_8^2$ & $-\zeta_8^2$ & $3$ & $0$ & $0$ & $-1$ & $-1$ & $1$ \\
\hline
$\chi_{12}$ & $0$ & $0$ & $0$ & $0$ & $-4$ & $1$ & $-1$ & $0$ & $0$ & $0$ \\
\hline
$\chi_{13}$ & $0$ & $0$ & $0$ & $0$ & $6$ & $0$ & $0$ & $0$ & $0$ & $-2$ \\
\hline
$\chi_{14}$ & $\zeta_8^3 - \zeta_8$ & $-\zeta_8^3 + \zeta_8$ & $-\zeta_8^3 + \zeta_8$ & $\zeta_8^3 - \zeta_8$ & $-6$ & $0$ & $0$ & $0$ & $0$ & $0$ \\
\hline
$\chi_{15}$ & $-\zeta_8^3 + \zeta_8$ & $\zeta_8^3 - \zeta_8$ & $\zeta_8^3 - \zeta_8$ & $-\zeta_8^3 + \zeta_8$ & $-6$ & $0$ & $0$ & $0$ & $0$ & $0$ \\
\hline
$\chi_{16}$ & $0$ & $0$ & $0$ & $0$ & $0$ & $2$ & $0$ & $2$ & $-1$ & $0$ \\
\hline
$\chi_{17}$ & $0$ & $0$ & $0$ & $0$ & $0$ & $2$ & $0$ & $-2$ & $1$ & $0$ \\
\hline
$\chi_{18}$ & $0$ & $0$ & $0$ & $0$ & $0$ & $-2$ & $0$ & $0$ & $0$ & $0$ \\
\hline
$\chi_{19}$ & $0$ & $0$ & $0$ & $0$ & $0$ & $0$ & $0$ & $2$ & $-1$ & $0$ \\
\hline
$\chi_{20}$ & $0$ & $0$ & $0$ & $0$ & $0$ & $0$ & $0$ & $-2$ & $1$ & $0$\\
\hline
\end{tabular}
\medskip
\caption{The character table of $G_{37}$.}\label{table:37}
\end{table}

\end{document}